\documentclass{my_gOPT2e}

\theoremstyle{plain}
\newtheorem{theorem}{Theorem}[section]

\theoremstyle{remark}

\theoremstyle{definition}
\newtheorem{definition}{Definition}

\def\a{{\alpha(\cdot,\cdot)}}
\def\b{{\beta(\cdot,\cdot)}}
\def\t{\tau}
\def\e{\epsilon}
\def\LI{{_aI_t^{\a}}}
\def\RI{{_tI_b^{\a}}}
\def\LDa{{_aD_t^{\a}}}
\def\LDb{{_aD_t^{\b}}}
\def\RDa{{_tD_b^{\a}}}
\def\RDb{{_tD_b^{\b}}}
\def\LC{{^C_aD_t^{\a}}}
\def\RCa{{^C_tD_b^{\a}}}
\def\RCb{{^C_tD_b^{\b}}}
\def\DC{{^CD_\gamma^{\a,\b}}}

\begin{document}

\title{Optimality Conditions for Fractional Variational Problems
with Dependence on a Combined Caputo Derivative of Variable Order}

\author{Dina Tavares$^{\rm a}$$^{\rm b}$,
Ricardo Almeida$^{\rm b}$$^{\ast}$\thanks{$^\ast$Corresponding author.
Email: ricardo.almeida@ua.pt\vspace{6pt}}
and Delfim F. M. Torres$^{\rm b}$\\\vspace{6pt}
$^{a}${\em{Polytechnic Institute of Leiria, 2410--272 Leiria, Portugal}};
$^{b}${\em{Center for Research and Development in Mathematics and Applications (CIDMA),
Department of Mathematics, University of Aveiro, 3810--193 Aveiro, Portugal}}\\
\received{Submitted Feb 5, 2014; Revised May 16, July 14 and Dec 16, 2014; Accepted Jan 8, 2015}}

\maketitle


\begin{abstract}
We establish necessary optimality conditions for variational
problems with a Lagrangian depending on a combined
Caputo derivative of variable fractional order. The
endpoint of the integral is free, and thus transversality conditions
are proved. Several particular cases are considered illustrating
the new results.

\begin{keywords}
dynamic optimization; fractional calculus;
variable fractional order; fractional calculus of variations.
\end{keywords}

\begin{classcode}
26A33; 34A08; 49K05.
\end{classcode}
\end{abstract}


\section{Introduction}

The fractional calculus of variations deals with optimization of functionals
that depend on some fractional operator \cite{book:FCV,book2:FCV}. This is a fast growing subject,
and different approaches have been developed by considering different types of Lagrangians,
e.g., depending on Riemann--Liouville or Caputo fractional derivatives,
fractional integrals, and mixed integer-fractional order operators
\cite{AlTorres,Atanackovic,Baleanu1,Baleanu2,Cresson,Rabei,Tarasov}.
The common procedure to address such fractional variational problems consists in solving
a fractional differential equation, called the Euler--Lagrange equation,
which every minimizer/maximizer of the functional must satisfy.
With the help of the boundary conditions imposed on the problem
at the initial time $t=a$ and at the terminal time $t=b$,
one solves, often with the help of some numerical procedure, the fractional differential
equation and obtain the possible solution to the problem \cite{MyID:225,MyID:258,RE1,RE2,RE3,RE4}.
When one of the boundary conditions is missing, that is, when the set of admissible functions
may take any value at one of the boundaries, then an auxiliary condition,
known as a transversality condition, needs to be obtained in order
to solve the fractional equation \cite{Alm:Malin}.
In this paper we do not only assume that $x(b)$ is free,
but the endpoint $b$ is also variable. Thus, we are not only interested
in finding an optimal curve $x(\cdot)$ but also the endpoint
of the variational integral, denoted in the sequel by $T$. This is known
in the literature as a free-time problem, and is already well
studied for integer-order derivatives (see, e.g., \cite{Chiang}).
Moreover, we consider a combined Caputo fractional derivative, that involves the left
and the right Caputo fractional derivative \cite{Malin:Tor,MyID:206,MyID:207}.
Also, the order of the fractional operator is not a constant,
and may depend on time \cite{MyID:280,Tatiana:IDOTA2011,Samko:1995}.
Very useful physical applications have given birth to the variable order fractional calculus,
for example in modeling mechanical behaviors \cite{MR3017585,MR2926061,Ref:ref3}.
Nowadays, variable order fractional calculus is particularly recognized as a useful
and promising approach in the modelling of diffusion processes, in order
to characterize time-dependent or concentration-dependent anomalous diffusion,
or diffusion processes in inhomogeneous porous media \cite{MR2926061}.
The importance of such variable order operators is due to the fact
that the behaviors of dynamical systems often change with their own evolution \cite{Ref:ref3}.
For numerical modeling of time fractional diffusion equations
we refer the reader to \cite{MR3017585}.
Our variational problem may describe better some nonconservative physical phenomena,
since fractional operators are nonlocal and contain memory
of the process \cite{Caputo,Riewe96,Riewe97}. We believe that it is reasonable
that the order of the fractional operators is influenced also by
the acting dynamics, which varies with time, and we trust that the results
now obtained can be more effective to provide a mathematical framework
to some complex dynamical problems \cite{Coimbra,Coimbra2,Ramirez2,Sheng,Sun}.

The paper is organized in the following way. In Section~\ref{sec:FC}
we introduce the basic concepts of the combined variable-order fractional calculus,
and we recall two formulas of fractional integration by parts (Theorem~\ref{thm:FIP}).
The problem is then stated in Section~\ref{sec:theorems}, consisting
of the variational functional
$$
\mathcal{J}(x,T)=\int_a^T L\left(t, x(t), \DC x(t)\right)dt+\phi(T,x(T)),
$$
where $\DC x(t)$ stands for the combined Caputo fractional derivative
of variable fractional order (Definition~\ref{def1}), subject to the
boundary condition $x(a)=x_a$. Our goal is to find necessary optimality conditions
that every extremizer $(x,T)$ must satisfy.
The main results of the paper provide necessary optimality conditions of Euler--Lagrange type,
described by fractional differential equations of variable order,
and different transversality optimality conditions
(see Theorems~\ref{teo1} and \ref{teo2}). Some particular cases of interest
are considered in Section~\ref{sec:part:cases}.
We end with two illustrative examples (Section~\ref{sec:ex}).


\section{Fractional calculus of variable order}
\label{sec:FC}

In this section we present the fundamental notions of the fractional calculus of variable order.
We consider the order of the derivative and of the integral to be a continuous
function $\alpha(\cdot,\cdot)$ with domain $[a,b]^2$, taking values on the open interval $(0,1)$.
Following \cite{Tatiana:IDOTA2011}, given a function $x:[a,b]\to\mathbb{R}$ we define:
\begin{itemize}
\item the left and right Riemann--Liouville fractional integrals of order $\a$ by
$$
\LI x(t)=\int_a^t \frac{1}{\Gamma(\alpha(t,\t))}(t-\t)^{\alpha(t,\t)-1}x(\t)d\t
$$
and
$$
\RI x(t)=\int_t^b\frac{1}{\Gamma(\alpha(\t,t))}(\t-t)^{\alpha(\t,t)-1} x(\t)d\t,
$$
respectively;
\item the left and right Riemann--Liouville fractional derivatives of order $\a$ by
$$
\LDa x(t)=\frac{d}{dt}\int_a^t \frac{1}{\Gamma(1-\alpha(t,\t))}(t-\t)^{-\alpha(t,\t)}x(\t)d\t
$$
and
$$
\RDa x(t)=\frac{d}{dt}\int_t^b\frac{-1}{\Gamma(1-\alpha(\t,t))}(\t-t)^{-\alpha(\t,t)} x(\t)d\t,
$$
respectively;
\item the left and right Caputo fractional derivatives of order $\a$ by
$$
\LC x(t)=\int_a^t\frac{1}{\Gamma(1-\alpha(t,\t))}(t-\t)^{-\alpha(t,\t)}x^{(1)}(\t)d\t
$$
and
$$
\RCa x(t)=\int_t^b\frac{-1}{\Gamma(1-\alpha(\t,t))}(\t-t)^{-\alpha(\t,t)}x^{(1)}(\t)d\t,
$$
respectively.
\end{itemize}

Motivated by the combined fractional Caputo and Riemann--Liouville definitions
\cite{Malin:Tor,MyID:206,MyID:207}, we propose the following definitions.

\begin{definition}
\label{def1}
Let $\alpha, \, \beta: [a,b]^2\rightarrow(0,1)$ and $\gamma=(\gamma_1,\gamma_2) \in [0,1]^{2}$.
The combined Riemann--Liouville fractional derivative operator $D_\gamma^{\a,\b}$ is defined by
$$
D_\gamma^{\a,\b}=\gamma_1 \, \LDa + \gamma_2 \, \RDb,
$$
acting on $x\in C([a,b])$ in the following way:
$$
D_\gamma^{\a,\b}x(t)=\gamma_1 \, \LDa x(t)+\gamma_2 \, \RDb x(t).
$$
Analogously, the combined Caputo fractional derivative operator,
denoted by $^{C}D_\gamma^{\a,\b}$, is defined by
$$
^CD_\gamma^{\a,\b}=\gamma_1 \, \LC+\gamma_2 \, \RCb,
$$
acting on $x\in C^{1}([a,b])$ as expected:
$$
^{C}D_\gamma^{\a,\b}x(t)=\gamma_1 \, \LC x(t)+\gamma_2 \, \RCb x(t).
$$
\end{definition}

The following theorem is proved in \cite{Od} and is a generalization
of the standard fractional formula of integration by parts for a constant $\alpha$
(see, e.g., formula (2) of \cite{Alm:Malin}).

\begin{theorem}[Theorem 3.2 of \cite{Od}]
\label{thm:FIP}
If $x,y \in C^1[a,b]$, then the fractional integration by parts formulas
$$
\int_{a}^{b}y(t) \, \LC x(t)dt
=\int_a^b x(t) \, {\RDa}y(t)dt
+\left[x(t) \, {_tI_b^{1-\a}}y(t) \right]_{t=a}^{t=b}
$$
and
$$
\int_{a}^{b}y(t) \, {\RCa}x(t)dt
=\int_a^b x(t) \, {\LDa} y(t)dt
-\left[x(t) \, {_aI_t^{1-\a}}y(t)\right]_{t=a}^{t=b}
$$
hold.
\end{theorem}


\section{Necessary optimality conditions}
\label{sec:theorems}

Consider the norm defined on the linear space $ C^1([a,b])\times \mathbb{R}$ by
$$
\|(x,t)\|:=\max_{a\leq t \leq  b}|x(t)|+\max_{a\leq t \leq b}\left| \DC x(t)\right|+|t|.
$$
Let $D$ denote the subset $ C^1([a,b])\times [a,b]$ endowed with the norm $\|(\cdot,\cdot)\|$,
such that $\DC x(t)$ exists and is continuous on the interval $[a,b]$.

\begin{definition}
We say that $(x^\star,T^\star)\in D$ is a local minimizer to the functional
$\mathcal{J}:D\rightarrow \mathbb{R}$ if there exists some $\epsilon>0$ such that
$$
\forall (x,T)\in D \, : \quad \|(x^\star,T^\star)-(x,T)\|<\epsilon\Rightarrow J(x^\star,T^\star)\leq J(x,T).
$$
\end{definition}

Along the work, we denote by $\partial_i z$, $i\in \{1,2,3\}$, the partial derivative
of a function $z:\mathbb{R}^{3} \rightarrow\mathbb{R}$ with respect to its $i$th argument,
and by $L$ the \emph{Lagrangian} $L:C^{1}\left([a,b]\times \mathbb{R}^2 \right)\to\mathbb{R}$.
For simplicity of notation, we introduce the operator $[\cdot]_\gamma^{\alpha, \beta}$ defined by
$$
[x]_\gamma^{\alpha, \beta}(t)=\left(t, x(t), \DC x(t)\right).
$$

Consider the following problem of the calculus of variations:
find the local minimizers of the functional $\mathcal{J}:D\rightarrow \mathbb{R}$, with
\begin{equation}
\label{funct1}
\mathcal{J}(x,T)=\int_a^T L[x]_\gamma^{\alpha, \beta}(t) dt + \phi(T,x(T)),
\end{equation}
over all $(x,T)\in D$ satisfying the boundary condition $x(a)=x_a$,
for a fixed $x_a\in \mathbb{R}$. The terminal time $T$ and terminal state $x(T)$ are free.
The \emph{terminal cost function} $\phi:[a,b]\times \mathbb{R}\to\mathbb{R}$ is at least of class $C^1$.

The next theorem gives fractional necessary optimality conditions to problem \eqref{funct1}.
In the sequel, we need the auxiliary notation of the dual fractional derivative:
$$
D_{\overline{\gamma}}^{\b,\a}=\gamma_2 \, {_aD_t^{\b}}
+\gamma_1 \, {_tD_T^{\a}},
\quad \mbox{where} \quad \overline{\gamma}=(\gamma_2,\gamma_1).
$$

\begin{theorem}
\label{teo1}
Suppose that $(x,T)$ is a local minimizer to the functional \eqref{funct1} on $D$.
Then, $(x,T)$ satisfies the fractional Euler--Lagrange equations
\begin{equation}
\label{ELeq_1}
\partial_2 L[x]_\gamma^{\alpha, \beta}(t)
+D{_{\overline{\gamma}}^{\b,\a}}\partial_3 L[x]_\gamma^{\alpha, \beta}(t)=0,
\end{equation}
on the interval $[a,T]$, and
\begin{equation}
\label{ELeq_2}
\gamma_2\left({\LDb}\partial_3 L[x]_\gamma^{\alpha, \beta}(t)
-{ _TD{_t^{\b}}\partial_3 L[x]_\gamma^{\alpha, \beta}(t)}\right)=0,
\end{equation}
on the interval $[T,b]$. Moreover, $(x,T)$ satisfies the transversality conditions
\begin{equation}
\label{CT1}
\begin{cases}
L[x]_\gamma^{\alpha, \beta}(T)+\partial_1\phi(T,x(T))+\partial_2\phi(T,x(T))x'(T)=0,\\
\left[\gamma_1 \, {_tI_T^{1-\a}} \partial_3L[x]_\gamma^{\alpha, \beta}(t)
-\gamma_2 \, {_TI_t^{1-\b}} \partial_3 L[x]_\gamma^{\alpha, \beta}(t)\right]_{t=T}
+\partial_2 \phi(T,x(T))=0,\\
\gamma_2 \left[ {_TI_t^{1-\b}}\partial_3 L[x]_\gamma^{\alpha, \beta}(t)
-{_aI_t^{1-\b}\partial_3L[x]_\gamma^{\alpha, \beta}(t)}\right]_{t=b}=0.
\end{cases}
\end{equation}
\end{theorem}

\begin{proof}
Let $(x,T)$ be a solution to the problem and
$\left(x+\e{h},T+\e\Delta{T}\right)$ be an admissible variation,
where $h\in C^1([a,b])$ is a perturbing curve,
$\triangle T \in\mathbb{R}$ represents an arbitrarily chosen
small change in $T$ and $\e \in\mathbb{R}$ represents
a small number $\left(\e\rightarrow 0\right)$.
The constraint $x(a)=x_a$ implies that all admissible variations
must fulfill the condition $h(a)=0$.
Define $j(\cdot)$ on a neighborhood of zero by
\begin{equation*}
\begin{split}
j(\e) &=\mathcal{J}(x+\e h,T+\e \triangle T)\\
&=\int_a^{T+\e \triangle T} L[x+\e h]_\gamma^{\alpha, \beta}(t)\,dt
+\phi\left(T+\e \triangle T,(x+\e h)(T+\e \triangle T)\right).
\end{split}
\end{equation*}
The derivative $j'(\e)$ is
\begin{equation*}
\begin{split}
j'(\e)&=\int_a^{T+\e \triangle T} \left(
\partial_2 L[x+\e h]_\gamma^{\alpha, \beta}(t) h(t)
+ \partial_3 L[x+\e h]_\gamma^{\alpha, \beta}(t) \DC h(t) \right)dt\\
&\quad +L[x + {\e h}]_\gamma^{\alpha, \beta}(T + \e \Delta T)\Delta T
+ \partial_1 \phi\left(T+\e \triangle T,
(x+\e h)(T+\e \triangle T)\right) \, \Delta T\\
&\quad +\partial_2\phi\left(T+\e \triangle T,
(x+\e h)(T+\e \triangle T)\right) \, (x+\e h)'(T+\e \triangle T).
\end{split}
\end{equation*}
Considering the differentiability properties of $j$, a necessary condition
for $(x,T)$ to be a local extremizer is given by
$\left. j'(\e) \right|_{\e=0}=0$, that is,
\begin{multline}
\label{eq_derj}
\int_a^T \left( \partial_2 L[x]_\gamma^{\alpha, \beta}(t) h(t)
+ \partial_3 L[x]_\gamma^{\alpha, \beta}(t) \DC h(t) \right)dt
+L[x]_\gamma^{\alpha, \beta}(T)\Delta T\\
+\partial_1 \phi \left(T, x(T)\right)\Delta T
+ \partial_2\phi(T, x(T))\left[h(t)+x'(T) \triangle T \right]=0.
\end{multline}
The second addend of the integral function \eqref{eq_derj},
\begin{equation}
\label{term}
\int_a^T \partial_3 L[x]_\gamma^{\alpha, \beta}(t) \DC h(t) dt,
\end{equation}
can be written, using the definition of combined
Caputo fractional derivative, as
\begin{equation*}
\begin{split}
\int_a^T & \partial_3 L[x]_\gamma^{\alpha, \beta}(t) \DC h(t) dt\\
&=\int_a^T \partial_3 L[x]_\gamma^{\alpha, \beta}(t)\left[\gamma_1
\, \LC h(t)+\gamma_2 \, \RCb h(t)\right]dt\\
&=\gamma_1 \int_a^T \partial_3 L[x]_\gamma^{\alpha, \beta}(t)\LC h(t)dt  \\
&\quad + \gamma_2 \left[ \int_a^b \partial_3 L[x]_\gamma^{\alpha, \beta}(t) \RCb h(t)dt
- \int_T^b \partial_3 L[x]_\gamma^{\alpha, \beta}(t) \RCb h(t)dt \right].
\end{split}
\end{equation*}
Integrating by parts (see Theorem~\ref{thm:FIP}), and since $h(a)=0$,
the term \eqref{term} can be written as
\begin{equation*}
\begin{split}
\gamma_1 & \left[ \int_a^T h(t) _tD_T^{\a} \partial_3
L[x]_\gamma^{\alpha, \beta}(t) dt
+\left[h(t) {_t I_T^{1-\a}\partial_3 L[x]_\gamma^{\alpha, \beta}(t)}\right]_{t=T} \right]\\
&+ \gamma_2 \Biggl[ \int_a^b h(t){_aD_t^{\b}} \partial_3 L[x]_\gamma^{\alpha, \beta}(t) dt
- \left[h(t){_a I_t^{1-\b}\partial_3 L[x]_\gamma^{\alpha, \beta}(t)}\right]_{t=b} \\
&\qquad \quad -\Biggl( \int_T^b h(t){_TD_t^{\b}} \partial_3 L[x]_\gamma^{\alpha, \beta}(t) dt
- \left[h(t){_TI_t^{1-\b}\partial_3 L[x]_\gamma^{\alpha, \beta}(t)}\right]_{t=b}\\
&\qquad \qquad \quad +\left[h(t){_TI_t^{1-\b}
\partial_3 L[x]_\gamma^{\alpha, \beta}(t)}\right]_{t=T} \Biggr) \Biggr].
\end{split}
\end{equation*}
Unfolding these integrals, and considering the fractional operator
$D_{\overline{\gamma}}^{\b,\a}$  with
$\overline{\gamma}=(\gamma_2,\gamma_1)$,
then \eqref{term} is equivalent to
\begin{equation*}
\begin{split}
\int_a^T h(t) & D_{\overline{\gamma}}^{\b,\a}\partial_3L[x]_\gamma^{\alpha, \beta}(t)dt\\
&+ \int_T^b\gamma_2h(t)\left[_aD_t^{\b}\partial_3 L[x]_\gamma^{\alpha, \beta}(t)
-{_TD_t^{\b}\partial_3L[x]_\gamma^{\alpha, \beta}(t)}\right]dt\\
&+\left[h(t)\left(\gamma_1 \, {_tI_T^{1-\a}\partial_3L[x]_\gamma^{\alpha, \beta}(t)}
-{\gamma_2 \, {_TI_t^{1-\b}\partial_3L[x]_\gamma^{\alpha, \beta}(t)}}\right)\right]_{t=T}\\
&+\left[ h(t)\gamma_2\left({_TI_t^{1-\b}\partial_3 L[x]_\gamma^{\alpha, \beta}(t)}
-{_aI_t^{1-\b}\partial_3L[x]_\gamma^{\alpha, \beta}(t)}\right)\right]_{t=b}.
\end{split}
\end{equation*}
Substituting these relations into equation \eqref{eq_derj}, we obtain
\begin{equation}
\label{eq_derj2}
\begin{split}
0= &\int_a^Th(t)\left[\partial_2 L[x]_\gamma^{\alpha, \beta}(t)
+ D_{\overline{\gamma}}^{\b,\a}\partial_3 L[x]_\gamma^{\alpha, \beta}(t)\right]dt\\
&+ \int_T^b \gamma_2 h(t) \left[_aD_t^{\b}\partial_3 L[x]_\gamma^{\alpha, \beta}(t)
-{_TD_t^{\b}\partial_3 L[x]_\gamma^{\alpha, \beta}(t)}\right]dt\\
&+ h(T)\left[\gamma_1 \, {_tI_T^{1-\a}\partial_3L[x]_\gamma^{\alpha, \beta}(t)}
-{\gamma_2 \, {_TI_t^{1-\b}\partial_3L[x]_\gamma^{\alpha, \beta}(t)}}
+\partial_2 \phi(t,x(t))\right]_{t=T}\\
&+\Delta T \left[ L[x]_\gamma^{\alpha, \beta}(t)+\partial_1\phi(t,x(t))
+\partial_2\phi(t,x(t))x'(t)  \right]_{t=T}\\
&+ h(b)\left[ \gamma_2  \left( _TI_t^{1-\b}\partial_3 L[x]_\gamma^{\alpha, \beta}(t)
-{_aI_t^{1-\b}\partial_3L[x]_\gamma^{\alpha, \beta}(t)}\right) \right]_{t=b}.
\end{split}
\end{equation}
As $h$ and $\triangle T$ are arbitrary, we can choose $\triangle T=0$ and $h(t)=0$,
for all $t\in[T,b]$, but $h$ is arbitrary in $t\in[a,T)$.
Then, for all $t\in[a,T]$, we obtain the first necessary condition \eqref{ELeq_1}:
$$
\partial_2 L[x]_\gamma^{\alpha, \beta}(t)
+D{_{\overline{\gamma}}^{\b,\a}}\partial_3 L[x]_\gamma^{\alpha, \beta}(t)=0.
$$
Analogously, considering $\triangle T=0$, $h(t)=0$,
for all $t\in[a,T]$, and $h$ arbitrary on $(T,b]$,
we obtain the second necessary condition \eqref{ELeq_2}:
$$
\gamma_2\left({\LDb}\partial_3L[x]_\gamma^{\alpha, \beta}(t)
-{ _TD{_t^{\b}}\partial_3L[x]_\gamma^{\alpha, \beta}(t)}\right)=0.
$$
As $(x,T)$ is a solution to the necessary conditions \eqref{ELeq_1}
and \eqref{ELeq_2}, then equation \eqref{eq_derj2} takes the form
\begin{equation}
\label{eq_derj3}
\begin{split}
0=&h(T)\left[\gamma_1 \, {_tI_T^{1-\a}\partial_3L[x]_\gamma^{\alpha, \beta}(t)}
-{\gamma_2 \, {_TI_t^{1-\b}\partial_3L[x]_\gamma^{\alpha, \beta}(t)}}
+\partial_2 \phi(t,x(t))\right]_{t=T}\\
&+\Delta T \left[ L[x]_\gamma^{\alpha, \beta}(t)+\partial_1\phi(t,x(t))
+\partial_2\phi(t,x(t))x'(t)  \right]_{t=T}\\
&+ h(b)\left[\gamma_2  \left( _TI_t^{1-\b}\partial_3 L[x]_\gamma^{\alpha, \beta}(t)
-{_aI_t^{1-\b}\partial_3L[x](t)}\right) \right]_{t=b}.
\end{split}
\end{equation}
Transversality conditions \eqref{CT1}
are obtained for appropriate choices of variations.
\end{proof}

In the next theorem, considering the same problem \eqref{funct1},
we rewrite the transversality conditions \eqref{CT1} in terms
of the increment on time $\Delta T$  and on the consequent increment
on $x$ and $\Delta x_T$, given by
\begin{equation}
\label{xt}
\Delta x_T = (x+h)\left( T + \Delta T \right) - x(T).
\end{equation}

\begin{theorem}
\label{teo2}
Let $(x,T)$ be a local minimizer to the functional \eqref{funct1} on $D$.
Then, the fractional Euler--Lagrange equations \eqref{ELeq_1}
and \eqref{ELeq_2} are satisfied together with
the following transversality conditions:
\begin{equation}
\label{CT2}
\begin{cases}
L[x]_\gamma^{\alpha, \beta}(T)+\partial_1\phi(T,x(T))\\
\qquad + x'(T) \left[ \gamma_2 {_TI}_t^{1-\b} \partial_3L[x]_\gamma^{\alpha, \beta}(t)
- \gamma_1 {_tI_T^{1-\a} \partial_3L[x]_\gamma^{\alpha, \beta}(t)} \right]_{t=T} =0,\\
\left[ \gamma_1\, {_tI_T^{1-\a}} \partial_3L[x]_\gamma^{\alpha, \beta}(t)
- \gamma_2\, {_TI_t^{1-\b}} \partial_3L[x]_\gamma^{\alpha, \beta}(t)\right]_{t=T}
+\partial_2 \phi(T,x(T))=0,\\
\gamma_2 \left[ _TI_t^{1-\b}\partial_3 L[x]_\gamma^{\alpha, \beta}(t)
-{_aI_t^{1-\b}\partial_3L[x]_\gamma^{\alpha, \beta}(t)}\right]_{t=b}=0.
\end{cases}
\end{equation}
\end{theorem}

\begin{proof}
The Euler--Lagrange equations are deduced following
similar arguments as the ones presented in Theorem \ref{teo1}.
We now focus our attention on the proof of the transversality conditions.
Using Taylor's expansion up to first order for a small $\Delta T$,
and restricting the set of variations to those for which $h'(T)=0$, we obtain
$$
(x+h)\left( T + \Delta T \right)=(x+h)(T)+x'(T) \Delta T + O(\Delta T)^{2}.
$$
Rearranging the relation \eqref{xt} allows us to express $h(T)$
in terms of $\Delta T$ and $\Delta x_T$:
$$
h(T)= \Delta x_T - x'(T) \Delta T + O(\Delta T)^{2}.
$$
Substitution of this expression into \eqref{eq_derj3} gives us
\begin{equation*}
\begin{split}
0=& \Delta x_T \left[\gamma_1 \, {_tI_T^{1-\a}\partial_3L[x]_\gamma^{\alpha, \beta}(t)}
-{\gamma_2 \, {_TI_t^{1-\b}\partial_3 L[x]_\gamma^{\alpha, \beta}(t)}}
+\partial_2 \phi(t,x(t))\right]_{t=T}\\
&+\Delta T \left[ L[x]_\gamma^{\alpha, \beta}(t)+\partial_1\phi(t,x(t))\right.\\
&\left.\qquad - x'(t)\left( \gamma_1 \, {_tI_T^{1-\a}\partial_3L[x]_\gamma^{\alpha, \beta}(t)}
-{\gamma_2 \, {_TI_t^{1-\b}\partial_3L[x]_\gamma^{\alpha, \beta}(t)}} \right)\right]_{t=T}\\
&+h(b)\left[\gamma_2\left( _TI_t^{1-\b}\partial_3L[x]_\gamma^{\alpha, \beta}(t)
-{_aI_t^{1-\b}\partial_3L[x]_\gamma^{\alpha, \beta}(t)}\right) \right]_{t=b}
+ O(\Delta T)^{2}.
\end{split}
\end{equation*}
Transversality conditions \eqref{CT2} are obtained
using appropriate choices of variations.
\end{proof}


\section{Particular cases}
\label{sec:part:cases}

Now, we specify our results to three particular
cases of variable terminal points.


\subsection{Vertical terminal line}

This case involves a fixed upper bound $T$. Thus, $\Delta T=0$ and,
consequently, the second term in \eqref{eq_derj3} drops out.
Since $\Delta x_T$ is arbitrary, we obtain the following
transversality conditions: if $T<b$, then
$$
\begin{cases}
\left[ \gamma_1 \, {_tI_T^{1-\a}} \partial_3L[x]_\gamma^{\alpha, \beta}(t)
- \gamma_2 \, {_TI_t^{1-\b}} \partial_3L[x]_\gamma^{\alpha, \beta}(t)\right]_{t=T}
+\partial_2 \phi(T,x(T))=0,\\
 \gamma_2 \left[ _TI_t^{1-\b}\partial_3L[x]_\gamma^{\alpha, \beta}(t)
 -{_aI_t^{1-\b}\partial_3L[x]_\gamma^{\alpha, \beta}(t)}\right]_{t=b}=0;
\end{cases}
$$
if $T=b$, then $\Delta x_T=h(b)$ and the transversality conditions reduce to
$$
\left[\gamma_1 \, {_tI_T^{1-\a}} \partial_3L[x]_\gamma^{\alpha, \beta}(t)
-\gamma_2 \, {_tI_T^{1-\b}\partial_3L[x]_\gamma^{\alpha, \beta}(t)}\right]_{t=b}
+\partial_2 \phi(b,x(b))=0.
$$


\subsection{Horizontal terminal line}

In this situation, we have $\Delta x_T=0$ but $\Delta T$ is arbitrary.
Thus, the transversality conditions are
$$
\begin{cases}
L[x]_\gamma^{\alpha, \beta}(T)+\partial_1\phi(T,x(T))\\
\qquad + x'(T) \left[ \gamma_2 {_TI_t^{1-\b} \partial_3L[x]_\gamma^{\alpha, \beta}(t)}
- \gamma_1 {_tI_T^{1-\a} \partial_3L[x]_\gamma^{\alpha, \beta}(t)} \right]_{t=T} =0,\\
\gamma_2 \left[ _TI_t^{1-\b}\partial_3L[x]_\gamma^{\alpha, \beta}(t)
-{_aI_t^{1-\b}\partial_3L[x]_\gamma^{\alpha, \beta}(t)}\right]_{t=b}=0.
\end{cases}
$$


\subsection{Terminal curve}

Now the terminal point is described by a given curve
$\psi:C^{1}([a,b])\rightarrow \mathbb{R}$,
in the sense that $x(T)=\psi (T)$. From Taylor's formula,
for a small arbitrary $\Delta T$, one has
$$
\Delta x(T)=\psi'(T) \Delta T+ O(\Delta T)^{2}.
$$
Hence, the transversality conditions are presented in the form
\begin{equation*}
\begin{cases}
L[x]_\gamma^{\alpha, \beta}(T)+\partial_1\phi(T,x(T))+ \partial_2 \phi(T, x(T))\psi'(T)\\
\ \ +\left(  x'(T)-\psi'(T)\right)
\left[ \gamma_2 \, {_TI_t^{1-\b} \partial_3L[x]_\gamma^{\alpha, \beta}(t)}
- \gamma_1 \, {_tI_T^{1-\a} \partial_3L[x]_\gamma^{\alpha, \beta}(t)} \right]_{t=T} =0,\\
\gamma_2 \left[ _TI_t^{1-\b} \partial_3L[x]_\gamma^{\alpha, \beta}(t)
-{_aI_t^{1-\b}\partial_3L[x]_\gamma^{\alpha, \beta}(t)}\right]_{t=b}=0.
 \end{cases}
\end{equation*}


\section{Examples}
\label{sec:ex}

In this section we show two examples for the main result of the paper.
Let $\alpha(t,\t)=\alpha(t)$ and $\beta(t,\t)=\beta(\t)$ be two functions
depending on a variable $t$ and $\t$ only, respectively.
Consider the following fractional variational problem: to minimize the functional
$$
\mathcal{J}(x,T)=\int_0^T\left[2 \alpha(t) -1+\left({^CD_\gamma^{\alpha(\cdot),\beta(\cdot)}} x(t)
-\frac{t^{1-\alpha(t)}}{2\Gamma(2-\alpha(t))}-\frac{(10-t)^{1-\beta(t)}}{2\Gamma(2-\beta(t))}\right)^2\right]dt
$$
for $t\in[0,10]$, subject to the initial condition $x(0)=0$ and where $\gamma = (\gamma_1,\gamma_2)=(1/2,1/2)$.
Simple computations show that for $\overline{x}(t)=t$, with $t\in[0,10]$, we have
$$
\DC \overline{x}(t)=\frac{t^{1-\alpha(t)}}{2\Gamma(2-\alpha(t))}+\frac{(10-t)^{1-\beta(t)}}{2\Gamma(2-\beta(t))}.
$$
For $\overline{x}(t)=t$ the functional reduces to
$$
\mathcal{J}(\overline{x},T)=\int_0^T\left(2 \alpha(t) -1\right)dt.
$$
In order to determine the optimal time $T$, we have to solve the equation
$2 \alpha(T) = 1$. For example, let $\alpha(t) = t^2/2$. In this case,
since $\mathcal{J}(x,T)\geq-2/3$ for all pairs $(x,T)$ and $\mathcal{J}(\overline{x},1)=-2/3$,
we conclude that the (global) minimum value of the functional is $-2/3$, obtained for $\overline{x}$ and $T=1$.
It is obvious that the two Euler--Lagrange equations \eqref{ELeq_1} and \eqref{ELeq_2}
are satisfied when $x=\overline{x}$, since
$$
\partial_3L[\overline{x}]_\gamma^{\alpha, \beta}(t)=0
\quad \mbox{for all} \quad t\in[0,10].
$$
Using this relation, together with
$$
L[\overline{x}]_\gamma^{\alpha, \beta}(1)=0,
$$
the transversality conditions \eqref{CT1} are also verified.

For our last example, consider the functional
$$
\mathcal{J}(x,T)=\int_0^T\left[2 \alpha(t) -1+\left({^CD_\gamma^{\alpha(\cdot),\beta(\cdot)}} x(t)
-\frac{t^{1-\alpha(t)}}{2\Gamma(2-\alpha(t))}-\frac{(10-t)^{1-\beta(t)}}{2\Gamma(2-\beta(t))}\right)^3\right]dt,
$$
where the remaining assumptions and conditions are as in the previous example.
For this case, $\overline{x}(t)=t$ and $T=1$ still satisfy
the necessary optimality conditions. However, it is now not obvious
that $(\overline x,1)$ is a local minimizer to the problem.


\section*{Acknowledgments}

This work is part of first author's Ph.D., which is carried out at the
University of Aveiro under the Doctoral Programme
\emph{Mathematics and Applications} of Universities of Aveiro and Minho.
It was supported by Portuguese funds through the
\emph{Center for Research and Development in Mathematics and Applications} (CIDMA),
and \emph{The Portuguese Foundation for Science and Technology} (FCT),
within project PEst-OE/MAT/UI4106/2014. Tavares was also supported
by FCT through the Ph.D. fellowship SFRH/BD/42557/2007;
Torres by EU funding under the 7th Framework Programme FP7-PEOPLE-2010-ITN,
grant agreement number 264735-SADCO, and by project PTDC/EEI-AUT/1450/2012,
co-financed by FEDER under POFC-QREN with COMPETE reference FCOMP-01-0124-FEDER-028894.
The authors are grateful to three referees
for valuable remarks and comments, which
significantly contributed to the quality of the paper.



\end{document}